\documentclass[11pt,reqno]{amsart}
\usepackage{amsmath,amssymb,amsthm, epsfig}
\usepackage{hyperref}
\usepackage[mathscr]{euscript}

\usepackage{comment}

\usepackage{xcolor}
\usepackage[nameinlink]{cleveref}
\hypersetup{
	colorlinks,
	linkcolor={blue!50!black},
	citecolor={blue!50!black},
	urlcolor={blue!50!black}
}

\usepackage{enumerate}
\usepackage{mathrsfs}
\usepackage{esint}

\newtheorem{theorem}{Theorem}
\newtheorem{definition}{Definition}
\newtheorem{lemma}{Lemma}
\newtheorem{proposition}{Proposition}

\newtheorem{remark}{Remark}

\date{}
\numberwithin{equation}{section}
\numberwithin{theorem}{section}
\numberwithin{lemma}{section}
\numberwithin{corollary}{section}
\numberwithin{remark}{section} 
\numberwithin{proposition}{section}
\numberwithin{definition}{section}

\newcommand{\R}{\mathbb{R}}
\newcommand{\supp}{\operatorname{supp } }
\newcommand{\diam}{\operatorname{diam}}

\newcommand{\loc}{\operatorname{loc}}

\def\XXint#1#2#3{{\setbox0=\hbox{$#1{#2#3}{\int}$ }
		\vcenter{\hbox{$#2#3$ }}\kern-.6\wd0}}

\begin{document}







\title[Sharp regularity for the anisotropic fractional Laplacian]{Sharp regularity for the anisotropic fractional Laplacian with Morrey data}

\author[R. Teymurazyan]{Rafayel Teymurazyan}
\address{Applied Mathematics and Computational Sciences (AMCS), Compu\-ter, Electrical and Mathematical Sciences and Engineering Division (CEMSE), King Abdullah University of Science and Technology (KAUST), Thuwal, 23955-6900, Kingdom of Saudi Arabia}{}
\email{rafayel.teymurazyan@kaust.edu.sa} 

\author[J.M.~Urbano]{Jos\'{e} Miguel Urbano}
\address{Applied Mathematics and Computational Sciences (AMCS), Compu\-ter, Electrical and Mathematical Sciences and Engineering Division (CEMSE), King Abdullah University of Science and Technology (KAUST), Thuwal, 23955-6900, Kingdom of Saudi Arabia; and CMUC, Department of Mathematics, University of Coimbra, 3000-143 Coimbra, Portugal}{} 
\email{miguel.urbano@kaust.edu.sa} 

\begin{abstract}
We obtain sharp regularity for the anisotropic fractional Laplace equation with a right-hand side belonging to a Morrey space. We derive local $C^{\sigma+\lambda-n}-$regularity of solutions when $\sigma+\lambda-n\neq1$, where $\sigma\in(0,2)$ is the order of the operator and $\lambda\in(n-\sigma,n)$ is the Morrey decay rate of the right-hand side. In the case $\sigma+\lambda-n=1$, we show that solutions are locally in the Zygmund class, and hence are Log-Lipschitz continuous. Explicit examples attest to the sharpness of the conditions.
\end{abstract}

\maketitle

\section{Introduction}
We study regularity for anisotropic fractional Laplace equations when the right-hand side has low integrability. More precisely, we consider equations of the form
\begin{equation}\label{main equation}
    Lu=f\quad\text{ in }\quad B_1,
\end{equation}
where
\begin{equation}\label{definition of L}
    Lu(x):=\text{P.V.}\int_{\R^n}\frac{u(x)-u(x+y)}{|y|^{n+\sigma}}a\left(\frac{y}{|y|}\right)\,dy.
\end{equation}
Here $\sigma\in(0,2)$ and $a\in L^\infty(\mathbb{S}^{n-1})$ is such that 
\begin{equation}\label{conditions on a}
    a(\theta)=a(-\theta)\quad\text{ and }\quad0<\Lambda\le a(\theta)\le\Lambda^{-1}<\infty,
\end{equation}
for a constant $\Lambda$. Equations of this type arise naturally in the study of symmetric stable L\'evy processes, and their regularity theory has been studied extensively due to a wide range of applications in probability, physics, finance, etc. (see, for example, \cite{CS09, CS11, CTU20, RO16, ROS16} and the references therein). In a generator of an $n-$dimensional pure jump L\'evy process, the kernel measures the frequency of jumps in the direction $y$. When $a(\cdot)$ is constant, the operator $L$ is a positive multiple of the fractional Laplacian. For $f\in L^\infty$, bounded weak solutions of \eqref{main equation} are known to be in $C^\sigma$, whenever $\sigma\neq1$ and in $C^{\sigma-\varepsilon}$, for any $\varepsilon>0$, when $\sigma=1$ (see \cite{ROS16}).

In contrast, we assume that $f$ is merely an integrable function satisfying a certain growth condition on its integrals over balls, i.e., it belongs to a suitable Morrey space
$$
f\in L^{1,\lambda}(B_1),\qquad n-\sigma<\lambda<n,
$$
meaning 
$$
\int_{B_r}|f|\lesssim r^\lambda.
$$
This assumption is substantially weaker than boundedness. It allows singular right-hand sides, excluding only concentrations that are too strong at small scales. The parameter $\lambda$ measures the local mass growth of $f$ and determines the amount of regularity that can be recovered from \eqref{main equation}. 

In \cite{F19}, it was shown that solutions of \eqref{main equation} are locally of class $C^{\sigma+\lambda-n-\varepsilon}$, for every sufficiently small $\varepsilon>0$, provided $\sigma+\lambda-n<1$. In this work, we go a step further, removing the $\varepsilon-$loss and deriving the corresponding optimal $C^{\sigma+\lambda-n}-$regularity of solutions. Moreover, for $\sigma+\lambda-n=1$, in \cite{F19} solutions were shown to be of class $C_{\loc}^{\beta}$, for any $\beta\in(0,1)$. We also improve this case, showing that the solutions are locally in the Zygmund class, which, in particular, implies log-Lipschitz regularity. For $\sigma+\lambda-n>1$, \cite[Corollary 6]{F19} already establishes the full $C^{\sigma+\lambda-n}_{\loc}-$regularity. Our direct potential argument uses properties of the kernel to recover the full regularity in all three regimes. An example shows that the assumption on $\lambda$ cannot be relaxed. The result is consistent with the corresponding result in the local framework, derived for the Poisson equation for uniformly elliptic divergence form operators in \cite{DiF92}, which, in turn, was recently extended to a broader class of operators in divergence form in \cite{DiFTU26}. 

This note is structured as follows. \Cref{s2} presents the precise mathematical setting and some known results for future reference. \Cref{s3} establishes regularity for the homogeneous equation. The main result is proved in \Cref{s4}, while \Cref{s5} confirms the sharpness of the conditions.

\section{Preliminaries}\label{s2}
In this section, we fix notation, recall definitions, and list several known results for future reference. In particular, we gather several properties of the potential kernel of the symmetric $\sigma-$stable process generated by the anisotropic fractional Laplacian. 

\subsection{Notation and definitions} Hereafter, $n\ge2$ and $B_r(x)$ is the ball of radius $r$ centered at $x$; we use $B_r:=B_r(0)$. With $a_{-}$ we denote any number strictly less than $a$. Throughout the paper, we use
\begin{equation}\label{definition of alpha}
    \alpha:=\sigma+\lambda-n,
\end{equation}
and
$$
C^\gamma:=C^{k,\gamma-k},\quad\text{where}\quad k=\lfloor\gamma\rfloor,\quad\gamma\ge0.
$$
We also use the Zygmund seminorm notation
$$
[u]_{*(B_r)}:=\sup_{\substack{\xi\neq0\\
    x,\,x\pm\xi\in B_r}}\frac{|u(x+\xi)+u(x-\xi)-2u(x)|}{|\xi|}.
$$
For $\sigma\in(0,2)$, we define
    $$
    L_\sigma:=\left\{u\in L^1_{\loc}(\R^n);\,\|u\|_{L_\sigma}:=\int_{\R^n}\frac{|u(y)|}{1+|y|^{n+\sigma}}\,dy<\infty
    \right\}.
    $$
Solutions of equations are understood in the distributional or weak sense, according to the definition below.
\begin{definition}
    A function $u\in L_\sigma$ is called a distributional solution of \eqref{main equation}, if 
    $$    \int_{\R^n}u(x)L\varphi(x)\,dx=\int_{B_1}f(x)\varphi(x)\,dx,\quad\forall\varphi\in C_0^\infty(B_1).
    $$
    If, in addition, $u\in H^{\sigma/2}(B_1)$, we call it a weak solution. 
\end{definition}
For the comparison principle for $L$ and further details, we refer the reader to, for example, \cite{CS09, PST25, S24, T26}. 
\begin{definition}
Let $\Omega\subset\R^n$ be a bounded domain with Lipschitz boundary, and let $\lambda\ge0$. The Morrey space $L^{1,\lambda}(\Omega)$ is the set of functions $f\in L^1(\Omega)$ for which
$$
\|f\|_{L^{1,\lambda}(\Omega)}
:=
\sup_{\substack{x\in\Omega\\0<r<\diam(\Omega)}}\left(
r^{-\lambda}\int_{\Omega\cap B_r(x)}|f(y)|\,dy\right)<\infty.
$$
\end{definition}
\begin{remark}
    Clearly, $L^{1,0}(\Omega)=L^1(\Omega)$. Furthermore, $L^{1,n}(\Omega)=L^\infty(\Omega)$ and $L^{1,\lambda}(\Omega)=\{0\}$, for $\lambda>n$ (see, for example, \cite[Theorem 5.5.1]{PKJF13}). For further properties of Morrey spaces, we refer the interested reader to \cite{M66, SDH1, SDH2}.
\end{remark}
\subsection{The potential kernel}
The heat kernel associated with $L$ is defined by
$$
p(x,t):=\frac{1}{(2\pi)^n}\int_{\R^n}e^{ix\cdot\xi}e^{-t\psi(\xi)}\,d\xi,\quad t>0,
$$
where
$$
\psi(\xi):=\int_{\R^n}\frac{(1-\cos(y\cdot\xi))}{|y|^{n+\sigma}}a\left(\frac{y}{|y|}\right)\,dy=c_\sigma\int_{\mathbb{S}^{n-1}}|\xi\cdot\theta|^\sigma a(\theta)\,d\theta,
$$
with
$$
c_\sigma:=\int_0^\infty\frac{1-\cos t}{t^{1+\sigma}}\,dt>0.
$$
In particular,
$$
    c|\xi|^\sigma\le\psi(\xi)\le C|\xi|^\sigma.
$$
Then the potential kernel, or fundamental solution for $L$, is 
$$
\Phi(x):=\int_0^\infty p(x,t)\,dt,\quad x\neq0,
$$
as
$$
L\Phi=\delta_0,
$$
in the distributional sense. Since
$$
\psi(r\xi)=r^\sigma\psi(\xi),\quad r>0,
$$
the potential kernel satisfies
\begin{equation}\label{Gamma is homogeneous}
    \Phi(rx)=r^{\sigma-n}\Phi(x),\quad r>0.
\end{equation}
The following result is from \cite[Theorem 1.1]{S10} (see also \cite{BS07}).
\begin{theorem}\label{t2.1}
    For every multiindex $\beta$ with $|\beta|<2\sigma$, there exists a constant $C>0$, depending only on $\beta$, $\sigma$, $\Lambda$ and $n$, such that
    \begin{equation}\label{derivative estimate}
        |D^\beta\Phi(x)|\le C|x|^{\sigma-n-|\beta|}.
    \end{equation}
    Moreover, if $2\sigma$ is not an integer, then $\Phi\in C_{\loc}^{2\sigma}(\R^n\setminus\{0\}).$
    If $2\sigma$ is an integer, then $\Phi\in C_{\loc}^{2\sigma_-}(\R^n\setminus\{0\})$.
\end{theorem}
\begin{remark}\label{r2.2}
    Observe that for $0<\delta<\min\{1,2\sigma\}$, whenever
    $$
    |x-y|\le\frac{1}{2}|x|,
    $$
    \Cref{t2.1} and \eqref{Gamma is homogeneous} yield
    $$
    |\Phi(x)-\Phi(y)|\le C|x-y|^\delta|x|^{\sigma-n-\delta}.
    $$
\end{remark}
The proof of the following result can be found in \cite[Corollary 4]{F19}.
\begin{proposition}\label{p2.1}
    Let $\sigma\in(0,2)$, $0<r<R$ and $f\in L^{1,\lambda}(B_R)$, where $n-\sigma<\lambda<n$. If $u\in H^{\sigma/2}(B_R)\cap L_\sigma$ is a weak solution of
    $$
    Lu=f\quad\text{in}\quad B_R,
    $$
    where $L$ is defined by \eqref{definition of L} with kernels satisfying \eqref{conditions on a}, then
    $$
    \|u\|_{C^\gamma(B_r)}\le C\left(\|u\|_{L^2(B_R)}+\|u\|_{L_\sigma}+\|f\|_{L^{1,\lambda}(B_R)}\right),
    $$
    for any $0<\gamma<\min\{1,\sigma+\lambda-n\}$. Here $C>0$ depends only on $\sigma$, $\lambda$, $\Lambda$, $\gamma$, $r$, $R$ and $n$.
\end{proposition}
The next result is from \cite[Corollary 6]{F19}.
\begin{proposition}\label{p2.2}
    Let $\sigma\in(1,2)$, $0<r<R$ and $f\in L^{1,\lambda}(B_R)$, where $n-\sigma+1<\lambda<n$. If $u\in H^{\sigma/2}(B_R)\cap L_\sigma$ is a weak solution of
    $$
    Lu=f\quad\text{in}\quad B_R,
    $$
    where $L$ is defined by \eqref{definition of L} with kernels satisfying \eqref{conditions on a}, then
    $$
    \|u\|_{C^{\sigma+\lambda-n}(B_r)}\le C\left(\|u\|_{L^2(B_R)}+\|u\|_{L_\sigma}+\|f\|_{L^{1,\lambda}(B_R)}\right),
    $$    
    where $C>0$ depends only on $\sigma$, $\lambda$, $\Lambda$, $r$, $R$ and $n$.
\end{proposition}

\section{Regularity for the homogeneous equation}\label{s3}

In this section, we show that distributional solutions of the homogeneous equation are in $H^{\sigma/2}_{\loc}(B_1)\cap C_{\loc}^{\gamma}(B_1)$, for any non-integer $\gamma\in(0,\sigma)$. \Cref{p2.1} establishes local $C^\gamma-$regularity for weak solutions in $H^{\sigma/2}_{\loc}(B_1)\cap L_\sigma$. Below, we extend it to the whole class of solutions in $L_\sigma$. In particular, we impose no a priori local Sobolev regularity on solutions. Instead, we work with distributional solutions in $L_\sigma$, and recover local energy regularity as a consequence. More precisely, we show that if $w\in L_\sigma$ is a distributional solution of $Lw=0$ in $B_1$, then $w\in H^{\sigma/2}_{\loc}(B_1)$. 
\begin{lemma}\label{important lemma on homogeneous}
    If $w\in L_\sigma$ solves 
    $$
    Lw=0\quad\text{in}\quad B_1
    $$
    in the sense of distributions, where $L$ is defined by \eqref{definition of L} with kernels satisfying \eqref{conditions on a}, then for $r\in(0,1)$ and $0<\gamma<\sigma$, $\gamma\neq1$, there exists $C>0$, depending only on $\gamma$, $r$, $\sigma$, $\Lambda$ and $n$, such that $w\in H^{\sigma/2}_{\loc}(B_1)\cap C_{\loc}^{\gamma}(B_1)$, and
    $$
    \|w\|_{C^\gamma(B_r)}\le C\|w\|_{L_\sigma}.
    $$
\end{lemma}
\begin{proof}
We first mollify $w$ and derive estimates that are uniform in the mollification parameter. More precisely, let $\eta\ge0$ be such that $\eta\in C_0^\infty(\R^n)$, $\supp\eta\subset B_1$ and
$$
\int_{\R^n}\eta(x)\,dx=1.
$$
Set 
$$
\eta_\varepsilon(x):=\frac{1}{\varepsilon^n}\eta\left(\frac{x}{\varepsilon}\right)
$$
and define
$$
w_\varepsilon(x):=(w*\eta_\varepsilon)(x):=\int_{\R^n}w(z)\eta_\varepsilon(x-z)\,dz=\int_{\R^n}\eta_\varepsilon(z)w(x-z)\,dz.
$$
Then
$$
|w_\varepsilon(x)|\le\|\eta_\varepsilon\|_\infty\int_{B_\varepsilon(x)}|w(z)|\,dz
$$
and $w_\varepsilon\in C^\infty(\R^n)$. There exists a constant, depending only on $n+\sigma$, such that if $|z|\le1$, then
$$
\frac{1}{1+|y+z|^{n+\sigma}}\le\frac{C}{1+|y|^{n+\sigma}}.
$$
Since $\eta_\varepsilon\ge0$, using Fubini's theorem, the change of variable $y=x-z$ and the previous inequality, we obtain
\begin{align}\label{weighted norm estimate}
    \|w_\varepsilon\|_{L_\sigma}&\le\int_{\R^n}\int_{\R^n}\frac{\eta_\varepsilon(z)|w(x-z)|}{1+|x|^{n+\sigma}}\,dz\,dx\nonumber\\
    &=\int_{\R^n}\eta_\varepsilon(z)\left[\int_{\R^n}\frac{|w(y)|}{1+|y+z|^{n+\sigma}}\,dy\right]\,dz\nonumber\\
    &\le C\|w\|_{L_\sigma}<\infty.
\end{align}
Furthermore,
\begin{equation}\label{equation for w varepsilon}    Lw_\varepsilon=0\quad\text{in}\quad B_{1-\varepsilon},
\end{equation}
in the sense of distributions. Also, since $w_\varepsilon$ is smooth (in particular, Lipschitz), for any $B\subset B_{1-\varepsilon}$, one has
\begin{align*}
    [w_\varepsilon]_{H^{\sigma/2}(B)}^2&:=\int_B\int_B\frac{|w_\varepsilon(x)-w_\varepsilon(y)|^2}{|x-y|^{n+\sigma}}\,dx\,dy\\
    &\le C|B|\int_0^{2(1-\varepsilon)}r^{1-\sigma}\,dr<\infty,
\end{align*}
and hence, $w_\varepsilon\in H^{\sigma/2}_{\loc}(B_{1-\varepsilon})$. Also, by the Dominated Convergence Theorem, $w_\varepsilon$ is, in fact, a weak solution of \eqref{equation for w varepsilon}. Hence, for each $0<\gamma<\min\{1,\sigma\}$, we can apply \Cref{p2.1} with 
$$
\lambda:=n-\frac{\sigma-\gamma}{2}
$$
and $0<s<t<1-\varepsilon$, yielding, for a $C>0$, depending only on $\gamma$, $s$ and $t$,
\begin{equation}\label{w varepsilon esimate}
    \|w_\varepsilon\|_{C^\gamma(B_s)} \leq C\left(\|w_\varepsilon\|_{L^2(B_t)}+\|w_\varepsilon\|_{L_\sigma}\right).
\end{equation}
For $r\in(0,1)$, set
$$
s:=\frac{1+r}{2},\quad t:=\frac{3+r}{4},
$$
and take $\varepsilon<\frac{1-r}{8}$. Using \eqref{weighted norm estimate}, from \eqref{w varepsilon esimate} we get
\begin{equation}\label{L infty estimate}
    \|w_\varepsilon\|_{L^\infty(B_s)} \leq C(t-s)^{-n}\left(\|w_\varepsilon\|_{L^2(B_t)}+\|w\|_{L_\sigma}\right).
\end{equation}
Employing \eqref{weighted norm estimate} one more time, we obtain
$$
\int_{B_t}|w_\varepsilon|^2\,dx\le\|w_\varepsilon\|_{L^\infty(B_t)}\int_{B_t}|w_\varepsilon|\,dx\le C\|w_\varepsilon\|_{L^\infty(B_t)}\|w\|_{L_\sigma},
$$
and therefore,
$$
    \|w_\varepsilon\|_{L^2(B_t)}\le C\|w_\varepsilon\|_{L^\infty(B_t)}^{1/2}\|w\|_{L_\sigma}^{1/2}.
$$
The latter, combined with \eqref{L infty estimate}, gives
\begin{equation}\label{L infty estimate2}
    \|w_\varepsilon\|_{L^\infty(B_s)} \leq C(t-s)^{-n}\left(\|w_\varepsilon\|_{L^\infty(B_t)}^{1/2}\|w\|_{L_\sigma}^{1/2}+\|w\|_{L_\sigma}\right).
\end{equation}
Fix any $\theta\in(0,1)$. Since for any $a,b\ge0$ one has 
$$
ab\le\theta b^2+\frac{a^2}{4\theta},
$$
using \eqref{L infty estimate2}, we arrive at
$$
\|w_\varepsilon\|_{L^\infty(B_s)} \leq \theta \|w_\varepsilon\|_{L^\infty(B_t)}+C_\theta(t-s)^{-2n}\|w\|_{L_\sigma},
$$
where $C_\theta>0$ is a constant depending on $C>0$ and $\theta>0$. Choosing $\theta=2^{-(2n+1)}$, and setting 
$$
r_k:=\frac{3+r}{4}-\frac{1-r}{4}2^{-k},\quad k\ge0,
$$
leads to
$$
\|w_\varepsilon\|_{L^\infty(B_s)}\le\theta^k\|w_\varepsilon\|_{L^\infty(B_{r_k})}+C_\theta(1-r)^{-2n}\|w\|_{L_\sigma}2^{6n}\sum_{j=0}^{k-1}(\theta2^{2n})^j.
$$
The first term in the last estimate tends to zero as $k\to\infty$, and the second term is bounded, since $\theta2^{2n}=\frac12$. Thus, letting $k\to\infty$, we conclude
\begin{equation}\label{L infty estimate3}
    \|w_\varepsilon\|_{L^\infty(B_s)}\le C\|w\|_{L_\sigma},
\end{equation}
for a constant $C>0$ depending only on $\sigma$, $\Lambda$, $r$ and $n$. Since $B_r\subset B_s$, using \eqref{w varepsilon esimate}, estimating the $L^2-$norm on the right-hand side by the uniform estimate \eqref{L infty estimate3}, we conclude 
$$
\|w_\varepsilon\|_{C^\gamma(B_r)}\le C\|w\|_{L_\sigma},
$$
where $C>0$ does not depend on $\varepsilon>0$. Passing to the limit as $\varepsilon\to0$, we arrive at
$$
\|w\|_{C^\gamma(B_r)}\le C\|w\|_{L_\sigma},
$$
for any $0<\gamma<\min\{1,\sigma\}$. 

\medskip

The desired estimate for $1<\gamma<\sigma$ holds as follows. Since $w_\varepsilon\in H_{\loc}^{\sigma/2}(B_{1-\varepsilon})\cap L_\sigma$ is a weak solution of \eqref{equation for w varepsilon}, we apply \Cref{p2.2} with 
$$
\lambda:=n+\gamma-\sigma,
$$
obtaining
\begin{align*}
    \|w_\varepsilon\|_{C^\gamma(B_r)}&\le C\left(\|w_\varepsilon\|_{L^2(B_s)}+\|w_\varepsilon\|_{L_\sigma}\right)\\
    &\le C\left(|B_s|^{1/2}\|w_\varepsilon\|_{L^\infty(B_s)}+\|w_\varepsilon\|_{L_\sigma}\right).
\end{align*}
Passing to the limit as $\varepsilon\to0$, and recalling \eqref{L infty estimate3}, we obtain the desired estimate.

\medskip

Thus, it remains to show that $w\in H^{\sigma/2}_{\loc}(B_1)$. For that purpose, choose
$$
\frac\sigma2<\gamma<\min\{1,\sigma\}.
$$
Since
$$
|w(x)-w(y)|\le C|x-y|^\gamma,
$$
then
\begin{align*}
    [w]_{H^{\sigma/2}(B)}^2&=\int_B\int_B\frac{|w(x)-w(y)|^2}{|x-y|^{n+\sigma}}\,dx\,dy\\
    &\le C|B|\int_0^2r^{2\gamma-\sigma-1}\,dr<\infty,
\end{align*}
i.e., $w\in H^{\sigma/2}_{\loc}(B_1)$.
\end{proof}

\section{Regularity for nonlocal equations}\label{s4}
In this section, we prove our main result. Hereafter, we extend $f$ by zero outside of $B_1$ if needed. Unless specified otherwise, $C>0$ is a universal constant, depending only on $\sigma$, $\lambda$, $\Lambda$ and $n$. We start by proving the following auxiliary result.
\begin{lemma}\label{auxiliary lemma}
    Let $\sigma\in(0,2)$, $x\in B_1$, $r\in(0,1)$,  and $\sigma+\lambda-n>m\ge0$. If $f\in L^{1,\lambda}(B_1)$, then
    $$
    \int_{B_r(x)}|x-y|^{\sigma-n-m}|f(y)|\,dy \leq C\|f\|_{L^{1,\lambda}(B_1)}r^{\sigma+\lambda-n-m},
    $$
    for a constant $C>0$ depending only on $\sigma$, $\lambda$, $m$ and $n$.
\end{lemma}
\begin{proof}
    Decompose
    $$
    B_r(x)=\bigcup_{j=0}^\infty A_j,
    $$
    where
    $$
    A_j:=\{2^{-j-1}r<|x-y|\le 2^{-j}r\}.
    $$
    Then, in $A_j$, one has
    \begin{align*}
        |x-y|^{\sigma-n-m}&<(2^{-j-1}r)^{\sigma-n-m}\\
        &=2^{m+n-\sigma}(2^{-j}r)^{\sigma-n-m}\\
        &=:C(2^{-j}r)^{\sigma-n-m}.
    \end{align*}
    Also,
    $$
    \int_{A_j}|f|\,dy\le\int_{B_{2^{-j}r}(x)}|f|\,dy\le(2^{-j}r)^\lambda\|f\|_{L^{1,\lambda}(B_1)}.
    $$
    Therefore,
    $$
    \int_{A_j}|x-y|^{\sigma-n-m}|f(y)|\,dy \leq C(2^{-j}r)^{\alpha-m}\|f\|_{L^{1,\lambda}(B_1)},   
    $$
    where $\alpha$ is defined by \eqref{definition of alpha}. Since $\alpha-m>0$,
    $$
    \sum_{j=0}^\infty2^{-j(\alpha-m)}<\infty,
    $$
    and the result follows.
\end{proof}
\begin{remark}
    Extending $f$ by $0$ outside of $B_1$, we deduce the conclusion of \Cref{auxiliary lemma} for any $r>0$ and $x\in\R^n$.
\end{remark}

We are now ready to prove the main result.

\begin{theorem}\label{main result}
    Let $\sigma\in(0,2)$ and $u\in L_\sigma$ be a distributional solution of 
    $$
    Lu=f\quad\text{ in }\quad B_1,
    $$
    where $f\in L^{1,\lambda}(B_1)$, and $n-\sigma<\lambda<n$. For every $\rho\in(0,1)$, there exists $C>0$, depending only on $\sigma$, $\lambda$, $\Lambda$, $\rho$ and $n$, such that
    \begin{itemize}
        \item If $\sigma+\lambda-n\neq1$, then         
        \begin{equation}\label{Holder estimate}
            \|u\|_{C^{\sigma+\lambda-n}(B_\rho)}\le C\left(\|u\|_{L_\sigma}+\|f\|_{L^{1,\lambda}(B_1)}\right).
        \end{equation}      
        \item If $\sigma+\lambda-n=1$, then $u$ is in the Zygmund class, i.e.,
        \begin{equation}\label{Zygmund estimate}
        \|u\|_{L^\infty(B_\rho)}+[u]_{*(B_\rho)}\le C\left(\|u\|_{L_\sigma}+\|f\|_{L^{1,\lambda}(B_1)}\right).
        \end{equation}        
    \end{itemize}    
\end{theorem}
\begin{proof}
We first prove the desired regularity result in a neighborhood of the origin. Let $\rho\in(0,1)$ and set 
\begin{equation}\label{definition of R and g}
    R:=\frac{1+\rho}{2},\qquad g:=f\chi_{B_R}.
\end{equation}
We then define
    $$
    v(x):=\int_{\R^n}\Phi(x-y)g(y)\,dy,
    $$
    where $\Phi$ is the fundamental solution defined in \Cref{s2}. Then
    $$
    Lv=g\quad\text{ in }\quad\R^n,
    $$
    in the sense of distributions.
    
    \bigskip
    
    \textsc{Step 1.} First, we make sure $v\in L_\sigma$. Indeed, using \eqref{derivative estimate} for $|\beta|=0$, one has
    \begin{equation}\label{upper control}
        |\Phi(x)|\le C|x|^{\sigma-n}.
    \end{equation}
    We write
    \begin{align}\label{4.1}
     \|v\|_{L_\sigma}&=\int_{|y|<10}\frac{|v(y)|}{1+|y|^{n+\sigma}}\,dy+\int_{|y|\ge 10}\frac{|v(y)|}{1+|y|^{n+\sigma}}\,dy\nonumber\\
     &=:I_1+I_2.
    \end{align}
    As
$$
g\in L^1(\R^n)\quad\text{ and }\quad\supp  g\subset \overline{B_R},
$$
using Fubini's theorem and \eqref{upper control}, we estimate
$$
I_1\le\int_{|y|<10}|v(y)|\,dy\le C\int_{B_R}|g(z)|\left(\int_{|y|<10}\frac{dy}{|y-z|^{n-\sigma}}\right)\,dz.
$$
For $z\in B_R$, setting $t:=y-z$ gives
$$
\int_{|y|<10}\frac{dy}{|y-z|^{n-\sigma}}\le\int_{|t|<10+|z|}\frac{dt}{|t|^{n-\sigma}}\le\int_{|t|<20}\frac{dt}{|t|^{n-\sigma}}=:C_1,
$$
and hence,
\begin{equation}\label{I1 estimate}
    I_1\le C_1\|g\|_{L^1(\R^n)}<\infty.
\end{equation}
For $|y|\ge10$, since $|z|<5$ for $z\in B_R$, we have
$$
|y-z|\ge|y|-|z|\ge|y|-5\ge\frac{1}{2}|y|,
$$
therefore,
$$
|v(y)|\le C \int_{B_R}\frac{|g(z)|}{|y-z|^{n-\sigma}}\,dz\le C 2^{n-\sigma} |y|^{\sigma-n}\|g\|_{L^1(\R^n)}=:\frac{C_2}{|y|^{n-\sigma}}.
$$
Thus,
$$
I_2\le C_2\int_{|y|\ge10}\frac{dy}{|y|^{n-\sigma}(1+|y|^{n+\sigma})}\le C_2\int_{|y|\ge10}\frac{dy}{|y|^{2n}}<\infty,
$$
which, combined with \eqref{I1 estimate} and \eqref{4.1}, yields
\begin{equation}\label{Lsigma of v}
    \|v\|_{L_\sigma} \leq C\|f\|_{L^{1,\lambda}(B_1)}<\infty.
\end{equation}
Set $w:=u-v\in L_\sigma$. Then
$$
Lw=0 \quad\text{in}\ B_R.
$$
By \Cref{important lemma on homogeneous}, for any $\gamma\in(0,\sigma)$, $\gamma\neq1$, we have
    $$
    \|w\|_{C^\gamma(B_\rho)}\le C\|w\|_{L_\sigma}.
    $$
    Hence, choosing $\gamma$ close enough to $\sigma$, we obtain higher regularity for $w$ than we intend to prove for $u$. Thus, it is enough to prove the desired regularity for $v$.

\bigskip

\textsc{Step 2.} Note that
$$
0<\alpha<2,
$$
where $\alpha$ is defined by \eqref{definition of alpha}. We consider three cases.

\medskip

\textsc{Case 1.} First, we show the result for
\begin{equation*}
    0<\alpha<1.
\end{equation*}
Take $x,z\in B_\rho$ and set $h:=|x-z|$. Then,
\begin{align}\label{v difference}
    v(x)-v(z)&=\int_{B_R}\left(\Phi(x-y)-\Phi(z-y)\right)f(y)\,dy\nonumber\\
    &=\left[\int_{B_R\cap B_{2h}(x)}+\int_{B_R\setminus B_{2h}(x)}\right]=:J_1+J_2.
\end{align}
Since $B_{2h}(x)\subset B_{4h}(z)$, using \eqref{upper control} and \Cref{auxiliary lemma} for $m=0$, we get, for a constant $C>0$,
\begin{align}\label{J1 estimate}
    |J_1|&\le C\int_{B_{2h}(x)}|x-y|^{\sigma-n}|f(y)|\,dy+C\int_{B_{4h}(z)}|z-y|^{\sigma-n}|f(y)|\,dy\nonumber\\
    &\le C\|f\|_{L^{1,\lambda}(B_1)}h^{\alpha}.
\end{align}
To estimate $J_2$, we first estimate the integral in the annulus 
\begin{equation}\label{Aj definition}
    A_j:=B_R\cap\left\{2^jh<|x-y|\le2^{j+1}h\right\},\quad j\ge1.
\end{equation}
Choosing
    $$
    \alpha<\delta<\min\{1,2\sigma\},
    $$
    and recalling \Cref{r2.2}, we obtain
   \begin{align*}
       \int_{A_j}|\Phi(x-y)-\Phi(z-y)||f(y)|\,dy&\le Ch^\delta(2^jh)^{\sigma-n-\delta}\int_{B_{2^{j+1}h}(x)}|f(y)|\,dy\\
       &\le Ch^\alpha\|f\|_{L^{1,\lambda}(B_1)}2^{j(\alpha-\delta)}.
   \end{align*}
    Therefore,
    \begin{align}\label{J2 estimate}
    |J_2|&\le\int_{B_R\setminus B_{2h}(x)}\left|\Phi(x-y)-\Phi(z-y)\right||f(y)|\,dy\nonumber\\
    &\le\sum_{j\ge1}\int_{A_j}\left|\Phi(x-y)-\Phi(z-y)\right||f(y)|\,dy\nonumber\\
    &\le C\|f\|_{L^{1,\lambda}(B_1)}h^\alpha\sum_{j\ge1}2^{j(\alpha-\delta)}<\infty.
    \end{align}
    Combining \eqref{v difference}, \eqref{J1 estimate} and \eqref{J2 estimate}, we arrive at
    $$
    |v(x)-v(z)|\le C\|f\|_{L^{1,\lambda}(B_1)}|x-z|^{\alpha},
    $$
    i.e., $v\in C^{\alpha}(B_\rho)$. Using \eqref{upper control} and \Cref{auxiliary lemma} with $m=0$, we get 
    $$
    \|v\|_{L^\infty(B_\rho)} \leq C\|f\|_{L^{1,\lambda}(B_1)},
    $$    
    where $C>0$ depends on $\sigma$, $\lambda$, $n$ and $\rho$. By translation, the argument above can be carried out around any fixed $x_0\in B_1$, and hence, $u\in C_{\loc}^{\alpha}(B_1)$, and \eqref{Holder estimate} holds.

    \medskip

\textsc{Case 2.} Assume now that $\alpha=1$. Take $x\in B_\rho$ and $\xi$ such that $x\pm\xi\in B_\rho$. Set
\begin{align}\label{second quotient}             \Delta_\xi^2v(x)&:=\int_{\R^n}\left[\Phi(x-y+\xi)+\Phi(x-y-\xi)-2\Phi(x-y)\right]g(y)\,dy\nonumber\\
&=\int_{|x-y|<2|\xi|}+\int_{|x-y|\ge2|\xi|}=:P+Q.
\end{align}
Observe that $B_{2|\xi|}(x)\subset B_{3|\xi|}(x+\xi)\cap B_{3|\xi|}(x-\xi)$. Using \eqref{upper control} and \Cref{auxiliary lemma} with $m=0$, for balls centered at $x$ and $x\pm\xi$, we obtain
\begin{equation}\label{P estimate}
    |P|\le C\sum_{b\in\{x,\,x\pm\xi\}}\int_{B_{3|\xi|}(b)}|b-y|^{\sigma-n}|g(y)|\,dy \leq C\|f\|_{L^{1,\lambda}(B_1)}|\xi|.
\end{equation}
Since $\sigma>1$, we can take $|\beta|=2$ in \Cref{t2.1} to get
\begin{equation}\label{Phi Hessian estimate}
    |D^2\Phi(\eta)|\le C|\eta|^{\sigma-n-2},\quad\eta\neq0.
\end{equation}
Now, if $\eta:=x-y$ and $|\eta|\ge2|\xi|$, then
$$
|\eta+s\xi|\ge\frac12|\eta|,\quad s\in(-1,1),
$$
and \eqref{Phi Hessian estimate} leads to
\begin{align*}
    |\Phi(\eta+\xi)+\Phi(\eta-\xi)-2\Phi(\eta)|&=\left|\int_{-1}^1(1-|s|)\xi^TD^2\Phi(\eta+s\xi)\xi\,ds\right|\\
    &\le C|\xi|^2|\eta|^{\sigma-n-2}.
\end{align*}
Integrating the latter over
$$
A_j:=\left\{2^j|\xi|\le|x-y|<2^{j+1}|\xi|\right\},\quad j\ge1,
$$
we get
\begin{align*}
    &\int_{A_j}|\Phi(\eta+\xi)+\Phi(\eta-\xi)-2\Phi(\eta)||g(y)|\,dy\nonumber\\
    &\le C|\xi|^2(2^j|\xi|)^{\sigma-n-2}\|f\|_{L^{1,\lambda}(B_1)}(2^{j+1}|\xi|)^\lambda\nonumber\\
    &=C|\xi|\|f\|_{L^{1,\lambda}(B_1)}2^{-j}.
\end{align*}
Hence,
\begin{equation}\label{Q estimate}
    |Q|\le C|\xi|\|f\|_{L^{1,\lambda}(B_1)}\sum_{j\ge1}2^{-j}=C|\xi|\|f\|_{L^{1,\lambda}(B_1)},
\end{equation}
and combining \eqref{second quotient}, \eqref{P estimate} and \eqref{Q estimate}, we obtain
\begin{equation}\label{quotient estimate for v}
    |\Delta_\xi^2v(x)|\le C|\xi|\|f\|_{L^{1,\lambda}(B_1)}.
\end{equation}
On the other hand, \Cref{important lemma on homogeneous} for $\gamma\in(1,\sigma)$ gives
$$
\|w\|_{C^\gamma(B_\rho)}\le C\|w\|_{L_\sigma}\le C\left(\|u\|_{L_\sigma}+\|f\|_{L^{1,\lambda}(B_1)}\right),
$$
and the Mean Value Theorem yields
\begin{align}\label{quotient estimate for w}
    |w(x+\xi)+w(x-\xi)-2w(x)|&\le|w(x+\xi)-w(x)|+|w(x-\xi)-w(x)|\nonumber\\
    \le&2\|Dw\|_{L^\infty(B_\rho)}|\xi|.
\end{align}
Using \eqref{upper control}, \Cref{auxiliary lemma} with $m=0$ and \eqref{Lsigma of v}, we conclude
$$
\|v\|_{L^\infty(\R^n)}+\|v\|_{L_\sigma}\le C\|f\|_{L^{1,\lambda}(B_1)},
$$
which combined with \eqref{quotient estimate for v} and \eqref{quotient estimate for w}, gives \eqref{Zygmund estimate}.

\medskip

\textsc{Case 3.} Finally, we consider the case when
$$
1<\alpha<2.
$$
Note that, in this case, $\sigma>1$, and \eqref{derivative estimate} with $|\beta|=1$ yields
\begin{equation}\label{estimate for gradPhi}
    |D\Phi(x)|\le C|x|^{\sigma-n-1}.
\end{equation}
    To obtain the desired H\"older estimate for $Dv$, we ensure that one can differentiate under the integral. Applying \Cref{auxiliary lemma} with $m=1$, we get
\begin{equation}\label{4.8}
    \int_{B_\rho(x)}|x-y|^{\sigma-n-1}|f(y)|\,dy\le C\|f\|_{L^{1,\lambda}(B_1)}\rho^{\alpha-1},
\end{equation}
i.e., the integral defining the candidate gradient is absolutely convergent. For any unit vector $e$, we aim to show that
\begin{equation}\label{4.9}
    \frac{v(x+te)-v(x)}{t}\longrightarrow \int_{B_R} e\cdot D\Phi(x-y)f(y)\,dy,\quad t\to0,
\end{equation}
where $R>0$ is defined in \eqref{definition of R and g}. Indeed,
\begin{align}\label{4.10}
    &\frac{1}{|t|}\int_{B_{2|t|}(x)}\left|\Phi(x+te-y)-\Phi(x-y)\right||f(y)|\,dy\\
    &\le\frac{C}{|t|}\left[\int_{B_{3|t|}(x+te)}|\Phi(x+te-y)||f(y)|\,dy+\int_{B_{2|t|}(x)}|\Phi(x-y)||f(y)|\,dy\right].\nonumber
\end{align}
Combining \eqref{upper control} and \Cref{auxiliary lemma} with $m=0$, one concludes that both integrals on the right-hand side of \eqref{4.10} are bounded by $C\|f\|_{L^{1,\lambda}(B_1)}|t|^{\alpha}$, and therefore, the right-hand side in \eqref{4.10} is bounded by 
$$
C\|f\|_{L^{1,\lambda}(B_1)}|t|^{\alpha-1}\to0,\quad|t|\to0.
$$
Here we used the fact that $\alpha>1$. Furthermore, in $|x-y|>2|t|$, \eqref{estimate for gradPhi} gives
$$
\frac{1}{|t|}\left|\Phi(x+te-y)-\Phi(x-y)\right|\le C|x-y|^{\sigma-n-1}.
$$
Recalling \eqref{4.8}, the fact that, in $B_R\setminus B_\rho(x)$, one has 
$$
|x-y|^{\sigma-n-1}\le \rho^{\sigma-n-1},
$$
and using the Dominated Convergence Theorem, we conclude that \eqref{4.9} holds.

\medskip

To obtain sharp H\"older regularity for $Dv$, setting $h:=|x-z|$, we split the integral and estimate each term as follows.
\begin{align}\label{Dv estimate}
    &|Dv(x)-Dv(z)|\nonumber\\
    &\le\int_{B_R}\left|D\Phi(x-y)-D\Phi(z-y)\right||f(y)|\,dy\nonumber\\
    &=\left[\int_{B_R\cap B_{2h}(x)}+\int_{B_R\setminus B_{2h}(x)}\right]=:A+B.
\end{align}
Combining \eqref{estimate for gradPhi} with the fact that $B_{2h}(x)\subset B_{4h}(z)$, and recalling \eqref{4.8}, we estimate
\begin{align}\label{A estimate}
    |A|&\le C\int_{B_{2h}(x)}|x-y|^{\sigma-n-1}|f(y)|\,dy+C\int_{B_{4h}(z)}|z-y|^{\sigma-n-1}|f(y)|\,dy\nonumber\\
    &\le C\|f\|_{L^{1,\lambda}(B_1)}h^{\alpha-1}.
\end{align}
To estimate $B$, observe that, for $|x-y|>2h$ and $\xi\in[x,z]$, one has
$$
|\xi-y|\ge\frac{1}{2}|x-y|.
$$
Using the Mean Value Theorem and \eqref{Phi Hessian estimate}, we arrive at
$$
|D\Phi(x-y)-D\Phi(z-y)|\le Ch|x-y|^{\sigma-n-2}.
$$
Hence, if $A_j$ is defined by \eqref{Aj definition}, then
\begin{align*}
    &\int_{A_j}|D\Phi(x-y)-D\Phi(z-y)||f(y)|\,dy\\
    &\le C\|f\|_{L^{1,\lambda}(B_1)}h(2^jh)^{\sigma-n-2}(2^{j+1}h)^\lambda\\
    &\le C\|f\|_{L^{1,\lambda}(B_1)}h^{\alpha-1}2^{j(\alpha-2)}.
\end{align*}
Therefore,
\begin{equation}\label{B estimate}
    |B|\le C\|f\|_{L^{1,\lambda}(B_1)}h^{\alpha-1}\sum_{j\ge1}2^{j(\alpha-2)}<\infty,
\end{equation}
where the last inequality follows from $\alpha<2$. Combining \eqref{Dv estimate}, \eqref{A estimate} and \eqref{B estimate}, we get
$$
|Dv(x)-Dv(z)|\le C\|f\|_{L^{1,\lambda}(B_1)}|x-z|^{\alpha-1}.
$$
Thus, $v\in C^{1,\alpha-1}(B_\rho)$, and therefore, $u$ enjoys the same regularity. Additionally, recalling \eqref{estimate for gradPhi} and using \Cref{auxiliary lemma} with $m=1$, we get
$$
\|Dv\|_{L^\infty(B_\rho)} \leq C\|f\|_{L^{1,\lambda}(B_1)},
$$
for a constant $C>0$ depending only on $\sigma$, $\lambda$, $\Lambda$, $n$ and $\rho$. As before, by translation, the argument can be carried out around any fixed $x_0\in B_1$, and hence, $u\in C_{\loc}^{\alpha}(B_1)$, and \eqref{Holder estimate} holds.
\end{proof}

\section{Examples and beyond}\label{s5}
In this section, we present examples that demonstrate the sharpness of the conditions. To start, observe first that the Zygmund class is a subset of log-Lipschitz continuous functions; therefore, for $\sigma+\lambda-n=1$, \Cref{main result} implies that solutions of \eqref{main equation} are locally log-Lipschitz continuous. Indeed, fix $x,z\in B_\rho$, $\rho\in(0,1)$, set 
$$
\rho<R:=\frac{1+\rho}{2}<1,\quad h:=z-x,
$$
and for $0<|h|<\frac{1-\rho}{8}$, choose $k\ge1$ such that
    $$
    \frac{1-\rho}{8}\le2^k|h|<\frac{1-\rho}{4}.
    $$
    Define
    $$
    u_h:=u(x+h)-u(x).
    $$
    Note that $x+2^jh\in B_R$, for any $j\in\{0,1,\ldots,k\}$, and
    $$
    2^{-k}\le\frac{8|h|}{1-\rho}\quad\text{and}\quad k\le C\left(1+\left|\log|h|\right|\right),
    $$    
    where $C>0$ depends on $\rho$. Using \eqref{Zygmund estimate} and iterating, one gets
    \begin{align*}
        |u_h|&=\frac12\left|u_{2h}-[u(x+2h)+u(x)-2u(x+h)]\right|\\
        &\le2^{1-k}\|u\|_{L^\infty(B_R)}+C\frac{k|h|}{2}\left(\|u\|_{L_\sigma}+\|f\|_{L^{1,\lambda}(B_1)}\right)\\
        &\le C\left(\|u\|_{L^\infty(B_R)}+\|u\|_{L_\sigma}+\|f\|_{L^{1,\lambda}(B_1)}\right)|h|\left(1+\left|\log|h|\right|\right).
    \end{align*}
    For $|h|\ge\frac{1-\rho}{8}$, the same estimate (with a constant depending on $\rho$) follows from $|u_h|\le2\|u\|_\infty$. Thus, for any $x,z\in B_\rho$ with 
    $$
    0<|x-z|\le\frac12,
    $$
    there exists a constant $C>0$, depending only on $\rho$, $\sigma$, $\Lambda$ and $n$, such that
    $$
    |u(x)-u(z)|\le C\left(\|u\|_{L_\sigma}+\|f\|_{L^{1,\lambda}(B_1)}\right)|x-z|\left(1+|\log|x-z||\right).
    $$        
    In particular,
    $$
    u\in C^{1^-}_{\loc}(B_1).
    $$

    \medskip
    
    \noindent \textbf{Example 1.} The following example shows that for $\sigma+\lambda-n=1$, solutions of \eqref{main equation} need not be Lipschitz continuous, yielding that the conclusion of \Cref{main result} is optimal.
    Let $\sigma\in(1,2)$,
    $$
    L:=(-\Delta)^{\sigma/2}\quad\text{and}\quad f(y):=\frac{y_1}{|y|^\sigma}\chi_{B_1}(y)\in L^{1,n+1-\sigma}(B_1).
    $$
    A solution of \eqref{main equation} is given in terms of the Riesz potential (see, for example, \cite{PST25, T26})
    $$
    u(x)=c_{n,\sigma}\int_{B_1}\frac{y_1}{|y|^\sigma|x-y|^{n-\sigma}}\,dy,
    $$
    where $c_{n,\sigma}>0$ is a constant depending only on $n$ and $\sigma$. Note that $u(0)=0$, and for $t>0$,
    $$
    u(te_1)=c_0t\log\left(\frac{1}{t}\right)+O(t),
    $$
    where $c_0:=c_{n,\sigma}|\mathbb{S}^{n-1}|(n-\sigma)/n>0$. Therefore,
    $$
    \frac{u(te_1)-u(0)}{t}\longrightarrow \infty,\quad\text{as}\quad t\to0,
    $$
    and hence $u$ is not Lipschitz.

    \medskip

    \noindent \textbf{Example 2.} The next example demonstrates the optimality of the H\"older regularity exponent $\sigma+\lambda-n$. Let $\alpha\in(0,\sigma)$, and $\alpha\neq1$ be defined by \eqref{definition of alpha}. Set 
    $$
    u(x):=\zeta(x)|x|^\alpha,
    $$
    where $\zeta\in C_0^\infty(B_1)$ and is identically $1$ on $B_{1/2}$. It is easy to check that if $(-\Delta)^{\sigma/2}u=:f$ in $B_1$, then
    $$
    |f(x)|\le C(1+|x|^{\alpha-\sigma}).
    $$
    Hence, $f\in L^{1,\lambda}(B_1)$. If $\alpha<1$, then $u(te_1)=|t|^\alpha$ in $B_{1/2}$ rules out every larger H\"older exponent below $1$. If $\alpha>1$, then $\partial_1u(te_1)=\alpha\,\text{sgn}(t)|t|^{\alpha-1}$ in $B_{1/2}$ excludes higher H\"older regularity for the gradient.

    \medskip

    \noindent \textbf{Example 3.} The established regularity is optimal in the sense that the condition $n-\sigma<\lambda$ cannot be relaxed. Indeed, if $\lambda=n-\sigma$ and 
    $$
    f(x):=|x|^{\lambda-n}\chi_{B_{1}}(x),
    $$
    then $f\in L^{1,\lambda}(B_1)$. Set 
    $$
    u:=(-\Delta)^{-\sigma/2}f.
    $$
    Then $u\in L_\sigma$ and
    $$
    (-\Delta)^{\sigma/2}u=f
    $$
    in the sense of distributions. However,
    $$
    u(x)\sim\log\frac{1}{|x|},\quad\text{ as }|x|\to0,
    $$
    i.e., $u$ is not even locally bounded. For $0<\lambda<n-\sigma$, the potential has a power singularity, as
    $$
    u(x)\sim|x|^{\sigma+\lambda-n},\quad\text{ as }\quad|x|\to0.
    $$

\bigskip

{\small \noindent{\bf Acknowledgments.} 
This publication is based upon work supported by King Abdullah University of Science and Technology (KAUST) under Award No. ORFS-CRG12-2024-6430.}

\end{document}